\documentclass[10pt,leqno]{amsart}

\usepackage{amsmath,amsopn,amssymb,amsthm,multicol}
\usepackage{amscd}
\usepackage{graphicx}
\usepackage{xcolor}
\usepackage{bm}

\numberwithin{equation}{section}
 \newtheorem{lemma} {Lemma} [section]

\newtheorem{prop} [lemma]{Proposition}

\title{Tunamis  on a deep open sea and on a gentle sloping beach 
 \\ 
 -- a mathematical theory -- }
 
\author{Tadayoshi Kano} 
\address{Institut Vercors, Ushinomiyatyo 13-4, Sakyoku-Yoshida, 6068302 KYOTO, Japan}
\email{institut\_vercors@cpost.plala.or.jp}

\dedicatory{Dedicated to Jean LERAY,  ancient patron of the author for his French government scholarship.}
 

\begin{document}

\begin{abstract}
Approaching a sloping beach, shallow water surface waves of Airy get suddenly  $ +\infty$ or $ -\infty$ propagation speed at the point of surface  $x = x_0$, say, where  the 
   tangent $\varGamma_x$ of  the surface  $y = \varGamma$ ``coincide'' with that  $b_x$  of the water-bottom $y = b(x)$,  losing the cruising sound speed of propagation so high on a deep open  sea.  That is,
 the tunamis gain instantaneously a  $ +\infty$  propagation speed just before  the crest as  
$(\varGamma_x - b_x)(x) \to +0$, $x \to x_0\!-\!0$ , and a $ -\infty$ propagation speed just after the trough as $(\varGamma_x - b_x)(x) \to -0$, $x \to x_0\!-\!0$.  We would have thus a big crush between 
the crest rushing forward and the trough rushing backward.  This is a mathematical  structure of tunamis ``on'' a sloping beach, in particular.
  \end{abstract}

  \maketitle

\section{Introduction}
\markboth{Kano}{Tunamis}
\def\thefootnote{\empty }
In this note, I'll give a simple mathematical theory for tunamis\footnote{ Partially supported by ESI.    An application of the author's invited talk in ESI  
conference:``Qualitative and numerical  aspects of water waves and other interface problems",  May 17, 18, 19 / 2011.
}. 
The tunamis should be shallow water waves, gentle on a deep open sea and very 
violent on a shallow water, on a sloping beach, for example. 

 We start from the fact that we have water waves, especially shallow water waves 
of Airy, already since seventies of last century, either as analytic functions solutions, 
or differentiable functions solutions (in Sobolev spaces). 

  My analysis therefore consists to clarify the cause of these sudden rushing 
of tunamis waves, as stated in the abstract, arriving on the shallow water near the coasts,
and, saying by the way, not the analysis of wave breaking itself on the shallow water. 

  Indeed, I know a specialist of tunamis study who cannot find the force who push inland
suddenly the tunamis waves  arriving on the shore who cease to move on the instant 
before start rushing. 

\medskip
 
\subsection{} What are tunamis in fact?  Tunamis are water waves produced typically by
the subsidence or the swell of the sea bottom caused by the submarine earthquakes.  
They are of very great wavelength and especially of the water moved from the 
bottom to the surface of the sea.  In the case of Hukusima ``tunamis", the swell of 
the sea bed is about 40 kilometers in the latitude and  60 kilometers in  longitude 
with 7$\sim$10 meters vertical elevation.  And the peak of seawater elevation caused 
by this swell is 7 meters with skirt of 400 kilometers, north/south, and 250 kilometers, 
east/west with the depth about 3900 meters.  Thus for these waves, the 
ratio of the water depth to the wavelength should be small and the ratio of the 
amplitude of waves to the water depth being finite could be, eventually, rather 
large.  We refer readers the following three papers and seabed map by Japan Coast 
Guard:

 (a) Toshiya Fujiwara et al.: The 2011 Tohoku-oki Earthquake: Displacement reaching 
      the Trench axis, Science, vol. 334 (2012.12.2), 1240--1241.
      
 (b) Yasuhiro Yoshida et al.: Source process of the 2011 off the Pacific coast of Tohoku
       Earthquake with the combination of teleseismic and strong motion data,  
       Earth Planets Space, 63 (2011), 565--569.
       
 (c) Tomohiro Takagawa: Inversion Methods for Tsunamis Source Estimation
      (in japanese),  ``Nagare", Journal of the Japanese Society of Hydrodynamics, 
      31(2013), 9--14.  (The author:  Asia-Pacific Center for Coastal Disaster Research, 
      Port and Airport Research Institute, Nagase 3-1-1, Yokosuka, Kanagawa, Japan). 
\medskip
      
      \subsection{} In the first place, I show that shallow water waves of Airy propagate with the 
``sound speed" on the deep open sea. And next, we show mathematically {\em the effect
of the diminution of the depth of water} transforming these shallow water waves of
Airy to a violent rushing tunamis ``on" the beaches losing their cruising sound speed 
on a deep open sea. For this purpose, we start with precise and exact definition of 
{\em shallow water surface waves}. 

\medskip

\section{Water surface waves on a shallow water}
\subsection{Equations of the problem and Friedrichs expansion}

Water waves in two dimensional  flow are governed by  system
of partial differential equations of Euler for the potential of flow $\varPhi(t, x, y)$  
and the wave profile $\varGamma =\varGamma(t,  x)$:
    \begin{eqnarray*}  & &  \varPhi_{x x}+ \varPhi_{y y}=0, 
     (x, y) \in \Omega_t=\{ (x, y) \in {\mathbb R}^2, \ b(x) < y < \varGamma(t, x), \ t > 0 \} \\ 
     & & 
  - b_x  \varPhi_x + \varPhi_y =0, \ x \in {\mathbb R}^1, \ y= b(x)  \\  
     & &  \varPhi_t + \frac{1}{2}({\varPhi_x}^2 + {\varPhi_y}^2 )  +y =0, \ 
    x \in {\mathbb R}^1, y= \varGamma(t, x)\\
 & &  \varGamma_t  + \varGamma_x \varPhi_x - \varPhi_y =0, \  x \in {\mathbb R}^1, \ y= \varGamma(t, x)
   \end{eqnarray*} 
   with initial data 
    \begin{eqnarray*}  & &     \varPhi(0, x, y) = \varPhi_0(x, y)    \\
 & &   \varGamma(0, x) = \varGamma_0(x)  > 0. 
   \end{eqnarray*}   

The general existence theorems for water surface waves were not known up to seventies of the last century. We knew only first linear approximation by Lagrange for infinitesimal oscillations 
as equations of vibrations of chords when G.B. Airy of Greenwich Observatory gave his first nonlinear approximation {\it of finite amplitude} [G.B. Airy: Tides and wave, B. Fellowes, 
London, 1845, 241-396] 
\begin{eqnarray*} &&
u_t + u u_x + \varGamma_x =0,  \quad u = \varPhi_x \\
&& 
\varGamma_t + ((\varGamma - b(x))u )_x = 0, 
 \end{eqnarray*}
which is known, now, as \emph{equations of shallow water waves of Airy}.

\medskip

\textbf{Remark}  We discuss especially the effect of the diminution of
the depth of water transforming these shallow water waves of Airy to a violent 
rushing tunamis ``on'' the beaches, we analyse so the water surface waves 
profiles on the plane perpendicular to the shoreline. Thus we work with above 
two dimensional flow. 

\medskip 

 In 1948, K.-O. Friedrichs has given a systematic method to drive this Airy's 
system by introduction of some non-dimensional parametre in the Euler equations
above and supposing a sort of asymptotic expansion for water waves  $\varPhi=\varPhi(t, x, y)$  and 
$\varGamma = \varGamma(t, x)$: CPAM vol.1 of Courant Institute of NY University. He recovered 
Airy's system for the first terms,  $\varPhi^0(t, x, y)$  and the wave profile $\varGamma^0(t, x)$,  of his 
``asymptotic'' expansion 
\begin{eqnarray*} &&
 \varPhi^0_t  + \frac{1}{2}{(\varPhi^0_x)}^2 + \varGamma^0 = 0, \\
&& 
\varGamma^0_t + {((\varGamma^0 - b(x))\varPhi^0_x )}_x = 0. 
 \end{eqnarray*}
 He himself did not give the mathematical justification for his ``asymptotic'' expansion.
 In 1979$\sim$83, we have shown a mathematical justification for this Friedrichs
expansion for water surface waves for analytic solutions \cite{Kan.-Nis.1}, \cite{Kan.-Nis. 2}. We started
by the recovery of Friedrichs' formation in 1948, formulas from (A.8) to (A.17),
in our dimensionless problem for water surface in the following.  

\medskip

 \subsection{ Dimensionless problem.}  
 
 Defining the dimensionless variables by
 \begin{eqnarray*} &&
 (t, x, y) = \left(\frac{\lambda}{c}t',  \lambda x', h y' \right), \quad  (\varGamma, \varPhi) = (h \varGamma', \  c \lambda \varPhi'), \quad \delta= \frac{h}{\lambda}    
   \end{eqnarray*}
  where  $\lambda$  is the {\it wave length},  $h$ the {\it depth of water} and $c$ is {\it the sound speed,} the square root of $gh$, $g$ being the {\it gravity},  equations for water waves become as follows dropping prime sign: 
  
    \begin{eqnarray}  & &    \delta^2 \varPhi_{x x}+ \varPhi_{y y}=0,  (x, y) \in \Omega_t=\{ (x, y) \in {\mathbb R}^2, \ b(x) < y < \varGamma(t, x), t > 0 \} \\  
     & &   - \delta^2 b_x \varPhi_x + \varPhi_y =0,  x \in {\mathbb R}^1, y= b(x) \\
   & &  \delta^2(\varPhi_t + \frac{1}{2}{\varPhi_x}^2  +y) + \frac{1}{2}{\varPhi_y}^2 =0,  x \in {\mathbb R}^1, y= \varGamma(t, x) \\
 & & \delta^2(\varGamma_t  + \varGamma_x \varPhi_x) - \varPhi_y =0,  
  x \in {\mathbb R}^1, y= \varGamma(t, x)
   \end{eqnarray} 
    for the velocity potential $\varPhi$ and wave profile function  $\varGamma$ with initial data

 \begin{eqnarray}  & &  \varPhi(0, x, y) = \varPhi_0(x, y)  \hspace{236pt}\quad  \\
 & &   \varGamma(0, x) = \varGamma_0(x)  > 0. \hspace{236pt}\quad 
   \end{eqnarray}   
   
   We have the existence theorem of analytic solution locally in time for analytic
initial data : \cite{Kan.-Nis.1},  \cite{Kan}. Our theorem in [2] is given explicitly for the flat bottom with 
$b_x = 0$,  but our result is easily extended to our present problem with the variable depth.

  In fact, the existence theorem of analytic solution to this bi-dimensional water surface equations 
  is to find a conformal mapping of the domain occupied by water with analytic water bottom and 
  an analytic free surface (surface waves). We have accomplished it in a scale of Banach spaces 
  of analytic functions  \cite{Kan.-Nis. 2}. 

  To solve surface waves equations (2.3)--(2.4),  $\varPhi_y$   on the surface  $y= \varGamma(t, x)$ 
should be represented by tangential derivative  $ \varPhi_x$   on the surface   $y= \varGamma(t, x)$.   
That is, we should calculate the so-called Dirichlet-Neumann map on the water surface.
 
  We applied, in fact, the Plemelj's formulas for the boundary values of analytic functions applying 
also the fundamental studies by L.C.Woods:

\noindent J. Plemelj: Problems in the sense of Klein and Riemann, Chapter 14, Interscience, 1964, 

\noindent  \underline{\phantom{J. Plemelj}}: Ein Erg\"anzungsatz zur Cauchyschen Integraldarstellung analytischer Funktionen,  Randwerte 
\phantom{J. Plemelj} \ betreffend, 18, XIX. Jahrgang 1908, 205--210,

\noindent  L.C. Woods: The theory of subsonic plane flow, Cambridge University Press, 1961.

\medskip

 We had the D-N map with explicit kernel function formula, and thus we had not only analytic solutions
 but also we see at once that they are indefinitely many times differentiable with respect to this 
 dimensionless parameter  $\delta= {h}/{\lambda} $    and in the sequel we had our Friedrichs
 expansion as follows in the next subsection. For the differentiable functions solutions, the problem 
 is not so  simple.

  \bigskip

 \subsection{ Friedrichs expansion for water surface wave and their equations.}  
 As K.-O. Friedrichs \cite{Frie} has proposed an asymptotic expansion for solutions of the 
 above dimensionless water waves Cauchy problem in 1948, and showed that the 
 first terms of this expansion satisfy Airy's famous shallow water waves equations, 
 we have effectively: 
 
 \begin{eqnarray} &&\varPhi^\delta(t, x, y) = \sum_{n=0}^\infty  \delta^{2n} \varPhi_n, 
\quad 
\varGamma^\delta(t, x, y) = \sum_{n=0}^\infty  \delta^{2n} \varGamma_n
  \end{eqnarray}
 and
 \begin{eqnarray} && \varPhi^0_t  + \frac{1}{2}{(\varPhi^0_x)}^2 + \varGamma^0 = 0, 
 \\ && 
 \varGamma^0_t + {((\varGamma^0 - b(x))\varPhi^0_x )}_x = 0. 
 \end{eqnarray}
  
  The coefficients of higher order terms satisfy, in fact, the following strictly hyperbolic systems with
  inhomogeneous terms for n=1,2,3,...:
  
   \begin{eqnarray} &&
   \varPhi^n _t  + { \varPhi^0_x \varPhi^n_x} + \varGamma^n =f{}_{n-1}, \\
&& 
 \varGamma^n_t + {(\varGamma^0 \varPhi^n_x+ \varGamma^n\varPhi^0_x)}_x = g{}_{n-1}. 
 \end{eqnarray}

  $f{}_{n-1}$  and  $g{}_{n-1}$  being differential polynomials with respect to $\{\varPhi{}_m,\varGamma{}_m\}_{m=0,1, 2,...,n-1}$,
  respectively.
  
  \medskip
  
   This Friedrichs expansion is justified mathematically to be Taylor expansion with respect to the dimensionless parameter $\delta$ for analytic solutions [3].  That is,  analytic 
solutions are infinitely many times differentiable with respect to $\delta$.

   For smooth functions solutions in Sobolev spaces,  so far, we have no such mathematical justification for  the Friedrichs expansion. But, let us hope to have it!
    
    \medskip

    \medskip
  
    {\bf Finite speed propagation.} Solutions for strictly hyperbolic system  (2.10)--(2.11) 
above, have finite dependence domain, that is they are of finite speed  
propagation. Thus these higher order terms of Friedrichs expansion for water waves 
(on shallow water) don't affect the propagation speed of tunamis.   This is the first condition for 
a justification of our Tunamis Equations  (3.5)--(3.6).
  
    \medskip
    
   {\bf Energy inequalities for (2.10)--(2.11).}
    It is well known that the strictly hyperbolic system  (2.10)--(2.11)  has an energy inequality for 
solutions with respect to  $L^2$  norm. The solutions are therefore completely controlled by initial 
data and the inhomogeneous terms.  They are finally completely controlled by solutions for 
(2.8)--(2.9).   This is the second condition for a justification for our system   (3.8)--(3.9)   as 
Tunamis Equations. 

    \medskip

     In addition to this, if we take into account of the fact that tunamis have very long 
wave-length, even near the beach, at least in a very early time, $\delta$  is very small and 
more for   $\delta^2 =\varepsilon$ . We have thus a good error estimate for the shallow water waves 
approximation for tunamis. 

    \medskip

\subsection{Shallow water  waves and Shallow water  waves of Airy}  
We precise now the definitions for these two which were sometimes in confusions.

\medskip

 {\bf Definition 2.1  Surface waves of shallow water}

\medskip

   For  $0 < \delta \ll 1$, the system  (2.1)--(2.4)  are surface waves of shallow water and their
solutions are shallow water surface waves. 

\medskip

{\bf Nota Bene}   Shallow water surface waves are not another than tidal waves of G.B. Airy,
 [H. Lamb: Hydrodynamics, $6^e$ \'ed., Cambridge, 1932, {\S}173].

\medskip

We know well that Airy himself has discovered the first nonlinear approximation
for shallow water surface waves (but for $ b_x = 0$) in the following form:[H. Lamb: ibid., {\S}187] :
  \begin{eqnarray} && u_t  + u u_x + \varGamma_x = 0, 
 \\ && 
 \varGamma_t + {((\varGamma - b(x))u )}_x = 0. 
 \end{eqnarray}
 
\medskip  

We repeat here that this approximation by Airy is now mathematically justified as 
the first terms of the Friedrichs expansion: Kano-Nishida for analytic solutions  \cite{Kan.-Nis.1}, \cite{Kan.-Nis. 2} 
and see  also \cite{Lan},  the book of D. Lannes, for differentiable functions solutions. 
However,  for the latter case  we have no  mathematical justifications for Friedrichs expansion itself. 
  We see in fact that  $u= \varPhi^0_x $  and  $\varGamma=\varGamma^0$ satisfy (2.12)--(2.13).
They would be, however, only a little part of shallow water surface waves. We define then  
``shallow water surface waves of Airy"  with these in mind as: 

\medskip

 {\bf Definition 2.2  Shallow water surface waves of Airy.}

\medskip
 
The system  (2.12)--(2.13)
above is {\it equations for shallow water surface waves of Airy} and their solutions
are {\it shallow water surface waves of Airy}. 

\medskip

  Then we can state the above discovery of Airy as follows: {\it shallow water surface waves} are  
approximated by {\it shallow water surface waves of Airy} with the error of order $\delta^2$ in 
suitable topology . 

\medskip

We now have  the following on the deep open sea:  
     
\begin{prop}  Shallow water surface waves of Airy propagate on a deep open 
sea with the cruising sound speed almost everywhere. 
\end{prop}

  We give the proof for this proposition in the next paragraph and here we 
show the following proposition as an application:

\begin{prop}   For analytic solutions of the system {\em (2.1)--(2.4)}, shallow water 
surface waves propagate with the cruising sound speed on a deep open sea.
\end{prop}

{\bf Nota Bene}  Everyone knows this sound speed propagation of water surface waves
which has not been, however, proved so far mathematically.
 
 \def\thefootnote{$*$}
\begin{proof} 
  The approximation by shallow water surface waves of Airy propagate by 
the cruising sound speed by the above Proposition 2.1, and the rest of this 
approximation propagate by finite speed satisfying strictly hyperbolic PDE  
shown by the Friedrichs expansion \cite{Kan.-Nis. 2}\footnote{In \cite{Kan.-Nis. 2}, Friedrichs expansion with $b_x=0$; here, if we  replace D-N map 
$\varGamma(t, x; \delta) = y(t, \xi_1; \delta) = \frac{A_\delta}{\delta} x(t, \xi_1; \delta)$ by 
$\varGamma(t, x; \delta) = y(t, \xi_1; \delta) = \frac{A_\delta}{\delta} x(t, \xi_1; \delta)+ D_\delta \beta(\xi_1)$, with $\xi_1=$ {\em inverse function of} $x=x(t, \xi_1; \delta)$,  $\beta(\xi)=$ {\em conformal image of water bottom} $y= b(x)$, the proof is also valid for the variable depth with 
 $D_\delta \beta(\xi_1) = \frac{1}{2 \delta}\int^{\infty}_{-\infty} \mathrm{sech} \frac{\pi}{2 \delta}(\tau -\xi) \beta (\tau)\, d \tau$. }.
\end{proof} 

  On a shallow water, however, {\it shallow water surface waves of Airy}  could not propagate always with
cruising sound speed but inevitably produce singularities. To see this, we derive TUNAMIS EQUATIONS 
(3.5)--(3.6)  from  (2.12)--(2.13)  and analyse them in the next §3. And finally in §4, we shall show mechanism
how these {\it shallow water surface waves of Airy} develop singularities as germs of violent inland and offshore tunamis near the coasts. 
   
{\bf Remark}
This Friedrichs expansion is proved in the Euler coordinates system. We remark here, by the way, that it is also proved in the Lagrangian coordinates system in T. Kano - S. Miki: {\em Sur les ondes superficielles de l'eau et le d\'eveloppement de Friedrichs dans le syst\`eme de coordonn\'ees de Lagrange},   Acte du Colloque International sur  La Th\'eorie des Equations aux D\'eriv\'ees Partielles et la Physique Math\'ematique, Paris (juin 2000), pp. 26--30. 

\medskip

Already mentioned above, however, differentiable functions solutions of water 
surface equations, we have no, not yet, mathematical justification for this Friedrichs 
expansion, and thus we have no general proof for sound speed propagation of water 
surface waves on a deep open sea.

\medskip

{\bf Nota Bene} It would be also interesting  to see shallow water waves from the point 
of view of wave structure in the connection of three essential parameters: 
{\em wave length, water depth and the amplitude of waves}.  In fact, water waves 
of the length $\lambda$ and of the amplitude  $a$  on the water of the depth $h$ satisfy-
ing the condition $\delta^2 \ll \varepsilon$ were called {\em shallow water waves}, 
where  $\delta$ and $\varepsilon$  is the ratio of the water depth to the wave length:  
$ \delta = h/\lambda$ and the ratio of the amplitude to the depth: $\varepsilon = a/h$, respectively \cite{Kan.-Nis. 2}, \cite{Kan}.   

For Hukusima tunamis of 11.3.2011,  $\delta^2 \sim (39/8000)^2$  or $\sim (39/5000)^2$  and  
$\varepsilon \sim 7/3900$, see ref.(c)  in the {\bf Introduction}.

\medskip

\section{Tunamis equations.  Proof of the Proposition 2.1.}  
\subsection{Tunamis equations.} 

We give first {\bf TUNAMIS EQUATIONS}. We rewrite   (2.12)--(2.13)   as follows:  
\begin{eqnarray} &&
u_t + u u_x + (\varGamma - b(x))_x  =  - b(x)_x,  \quad u = \varPhi_x \\
&& 
(\varGamma- b(x))_t + ((\varGamma - b(x))u )_x = 0.
 \end{eqnarray}
   
 Let us now define $\gamma$ by $\gamma^2 $=$ \varGamma-b(x) > 0$,
  $\gamma=\sqrt{\varGamma-b(x)} > 0$, we have then  
   \begin{eqnarray} 
   && P_t  + (\gamma+u)P_x = -b_x , 
 \\  
&& Q_t - (\gamma - u )Q_x= -b_x
\end{eqnarray}
for $P=u+2\gamma$  and  $Q=u-2\gamma$. 
 
 \medskip
 
  From these, we have finally the TUNAMIS EQUATIONS: 
  
 \medskip
  
 {\bf Definition 3.1  Tunamis equations.} The following system of partial differential equations are tunamis equations: 
  \begin{eqnarray} 
   &&  P_t  + \left(\gamma+u+\frac{b_x}{P_x}\right)P_x = 0 \\
  && Q_t - \left(\gamma-u-\frac{b_x}{Q_x} \right)Q_x =  0.
 \end{eqnarray}

 Let us discuss a little bit on this definition of tunamis equations: water surface wave  $P$, inland 
tunamis, propagate toward "beaches" with the speed $\displaystyle \gamma+u+\frac{b_x}{P_x} $,
modifying their cruising velocity $ \gamma+u=\sqrt{\varGamma-b(x)}+u$  by  $\displaystyle \frac{b_x}{P_x}$
referring the state of the sea-bed  $b_x$ in the connection with the structure $P_x$ of $P$ himself.
The same for the coastal tunamis $Q$. It is just from this structure with  $\displaystyle \frac{b_x}{P_x}$ 
and $\displaystyle \frac{b_x}{Q_x} $of this ``tunamis equations" start a violent tunamis development on ``a beach"  
as we see it in the {\S}4.
 
  Even if $b_x$ vanishes, there is certainly a wave collapse on a shore by the non-linearity, as Airy 
  insisted, but never a violent tunamis.  His theory of the shallow water waves is in  §187  of  H. Lamb, 
  as indicated already above, but for  $b_x = 0$. 
  
  \medskip
  
  \subsection{Tunamis equations and Airy's shallow water waves.} 
  We have no problems of existence theorems for  (3.1)--(3.2)  and  (3.3)--(3.4), but  tunamis equations 
  presented such as it is, it would not be assured that all tunamis solutions to arise again as Airy's 
  shallow water waves. If we concern, however, ourselves with tunamis satisfying $\displaystyle u=\frac{1}{2}(P+Q)$  and  $\displaystyle\gamma=\frac{1}{4}(P-Q)$, equations (3.3)--(3.4) for them give us 
 
\begin{eqnarray*} 
    &&  u_t + u u_x + {(\gamma^2 + b)}_x =0, \\
    && (\gamma^2)_t + \gamma^2 u _x + 2 u \gamma \gamma_x =0 
\end{eqnarray*}
and we recover Airy's shallow water waves equations with $\gamma^2 = \varGamma - b$. 

 We are thus going to analyse the behaviors of tunamis  $\{P, \, Q\}$ satisfying   
$\displaystyle  u=\frac{1}{2}\{P+ Q\} $ and $\displaystyle\gamma=\frac{1}{4}(P-Q)$, that is, tunamis  $P = u + 2 \gamma$ and $Q =  u - 2 \gamma$. 

\subsection{} Now we discuss the behavior of  $P$  on a deep open sea: $\gamma = \sqrt{\varGamma - b(x)} \gg 1$,  giving firstly a proof for the {\bf Proposition 2.1}:

 \medskip
 
\noindent {\bf Proof of the Proposition 2.1}. 

\smallskip

\noindent [I] For the sea-bed such that $b_x \neq 0$,  as long as we have $P_x \neq 0$, that is  as long 
as we have  $\displaystyle P_x = u_x + \frac{\varGamma_x - b_x}{\sqrt{\varGamma  - b }}\neq 0 $, there would be no visible affection on the 
cruising speed of ``tunamis" (shallow water waves) by conditions of sea-bed form. 

   In fact,  tunamis $P$  propagate by the speed
   $$\gamma +u+ \frac{b_x}{P_x} = \sqrt{\varGamma  -b} + u + 
   \frac{ b_x \sqrt{\varGamma  -b}}{u_x \sqrt{\varGamma  -b} + (\varGamma _x - b_x)}
   $$
   which is {\bf  a velocity of order of $\sqrt{\varGamma  -b} $    multiplied by a certain finite quantity}.
   
      That is, the tunamis  $P$   propagate by the sound speed, the gravity $g$  being
nondimensionalised as unity. 

 \medskip
 
\noindent [II]  On the other hand, if the tunamis  $P$: 
 $$P = u + 2 \gamma  = u +2 \sqrt{\varGamma  -b}
 $$                                   
satisfy  $ P_x(X) = 0$  on  $ x = X$  on the deep open sea, $\sqrt{\varGamma  -b} \gg 1$, where   $b_x (X) $  
doesn't vanish,  what would happen?  As the tunamis satisfy 
$$P_t + \left(\gamma +u + \frac{b_x}{P_x}\right)P_x = 0,$$
it is certain that the tunamis  $ P(t)$ become singular by this component  $\displaystyle \frac{b_x}{P_x} \to +\infty$  or $\to -\infty$ .  Now, we see that the velocity depends also on the sea 
current   $u(t, x)$, then we must see what happens there along the flow from $x= X - \epsilon$ to
$X$. To see this, we calculate $P_x(X) - P_x(X - \epsilon)$, $\epsilon > 0$:
 
 \medskip
 
$\displaystyle P_x(X)-P_x(X-\epsilon)=u_x(X)-u_x(X-\epsilon) +
\frac{(\varGamma_x - b_x)(X)}{\sqrt{(\varGamma  - b)(X)}} -
\frac{(\varGamma_x - b_x)(X-\epsilon)}{ \sqrt{(\varGamma - b)(X-\epsilon)} } 
=\epsilon u_{xx}(X-\delta\epsilon)+ +\epsilon\frac{(\varGamma_x - b_x)_x(X-\delta^{\prime}\epsilon)}{\sqrt{(\varGamma  - b)(X-\delta^{\prime}\epsilon)} }
-\frac{\epsilon}{2} \left[\frac{1}{\sqrt{(\varGamma  - b)(X-\delta^{\prime}\epsilon)}} \left( \frac{(\varGamma_x - b_x)(X-\delta^{\prime}\epsilon)}{\sqrt{(\varGamma  - b)(X-\delta^{\prime}\epsilon)}} \right)^2\right]$, 
for $0<\delta, \delta^\prime<1$. 

\medskip 

 In consequence, we see the following two possibilites: 
 
(A) \ For $0 < \epsilon \ll 1$, $P_x(X - \epsilon) < P_x(X) =0$ and $P_x(X - \epsilon)  \to -0$, as $\epsilon \to 0$, and thus the tunamis  $P$  get at $x=X-0$  an  instantaneous 
$-\infty$   speed rushing, in consequence, to the outer open sea. 

 (B) \ For $0 < \epsilon \ll 1$, $P_x(X - \epsilon) > P_x(X) =0$ and $P_x(X - \epsilon)  \to +0$, as $\epsilon \to 0$. Thus the tunamis
get an instantaneous $+\infty$ speed at  at $x=X-0$ rushing (possibly) thus to inland as inland 
tunamis.

  We are going to see the conditions for them:
  
   \medskip
   
($1^\circ$) Firstly, we examine the role of sea current  $u = u(t, x)$  on the deep open sea:
 $\sqrt{\varGamma  - b } \gg 1$.
Many studies show that these sea currents are of the speed of 2 or 3 knots,  i.e.  only
several  kilometers per hour even in equators currents or Gulf current.  Whereas the
water waves speed, $\sim \sqrt{\varGamma - b }$, is about $738 km/h$ on the open Pacific Ocean the mean depth of which is $4282m$ (the maximum depth is $8020m$).
 
   \medskip
   
 ($2^\circ$) On the other hand, $u_x $ is more important as   $u_x < 0$  perform the crest and $ u_x  > 0$  the
trough.  The signature of $ u_{xx} $ would be more important for any  $u_x$   positive or negative
for these performance of crest or trough.
 
   \medskip
   
 ($3^\circ$) \  Now it seems that the sea current  $u$  is not so great and it seems that many sea currents  as North-pacific circular current or the Oyasio current as almost stationary. And thus it seems that  $u_x$ is not so large.  We see then by $\displaystyle P_x = u_x + \frac{\varGamma_x - b_x}{\sqrt{\varGamma  - b }} = 0 $  that  $\displaystyle  \frac{\varGamma_x - b_x}{\sqrt{\varGamma  - b }}$ 
should be finite on the  $x = X$   where  $P_x   = 0$   which correct himself that   $(\varGamma_x - b_x)(X)$ is very large on this point  $X $ on the deep open sea.  Then except some singular case as 
the submarine-cliff of Japan Trench, if $P_x(X) = 0$  occurs actually, then  we must have
 $|\varGamma_x(X)| \gg 1$ on the  deep open sea $\sqrt{\varGamma  - b } \gg 1$. 
 
   Now, in these circumstances, it would be difficult to see by the concrete manner what 
occurs  with $P_x(X - \epsilon) < P_x(X) = 0$ or $P_x(X - \epsilon) > P_x(X) = 0$ on that point  $X$  where  $P_x(X) = 0$.
 
   \medskip

 ($4^\circ$) \ The situation is totally different actually from $P_x(X) = 0$ on the coast:  $\sqrt{\varGamma  - b } \ll 1$.
There,   $\displaystyle  \{P_x(X) = u_x(X)+ \frac{\varGamma_x - b_x}{\sqrt{\varGamma  - b }}=0 ,   \sqrt{\varGamma  - b }\ll 1$  and   $u_x$   is finite$\}$  allow us to give 
precisely the point  $X$  on which $P_x(X) =0$ occurs as the point  $X$ where we have   
$-u_x(X)\sqrt{\varGamma  - b } = \varGamma_x - b_x \sim \pm 0 $, $ + 0$ {\em near the crest and }  $-0$ {\em   near the trough}.
That is, on a shallow water of the sea (near the coasts especially),  $P_x(X) =0$, if it occurs
actually  on  $X$, on this  $X$  {\em the tangent of tunamis surface is actually very close to that of
the sea-bed}.  Here we would be able to analyze the mechanism of the ``birth'' of tunamis:
the transformation of arriving shallow water waves on the coasts to actual rushing inland 
or outer tunamis by the effect of the diminishing depth on the coasts. It will be discussed 
actually in the next paragraph ${\S}$4.
 
   \medskip

 ($5^\circ$) \ Anyway these agitations on exceptional points on the surface of tunamis, 
however, would not affect so much these tunamis  $P$  on {\em a deep open sea}: 
wave-lengths of tunamis  are so long and the sea is so deep that the possible crush of 
tunamis waves near the crest rushing toward and the tunamis waves near the trough 
rushing backward would not happen in reality on a deep open sea.

  \medskip
  \newpage 
 
\noindent [III] \ {\bf What does it happen when   $P_x$  vanishes on the points  $\bm x^*$ where  $\bm{b_x (x^*) = 0}$ ?}
 
  \medskip
  
 Let $b_x(x*)$ vanish as we see $ b_x(x*) =B_1{(x - x^*)}^p_{}(1 + O(| x - x^* |))$, $B_1 \neq 0$, around   $x = x^*$.  And   let  $P_x (x^*)$    vanish so as to be: 
 
 $\displaystyle P_x (x^*) =   u_x(x^*) + \frac{(\varGamma_x - b _x)(x^*)}{\sqrt{(\varGamma  - b)(x^*) } } $
 vanishes as $u_x \sqrt{\varGamma  - b } + \varGamma_x = C_1{(x-x^*)}^q_{}(1+ O(|x-x^*|))$ around $x = x^*$. We have then, for $ x \sim x^*$ 
 $$\frac{b_x}{P_x} = \frac{B_1{(x - x^*)}^p_{}(1+ O(|x-x^*|))\sqrt{\varGamma  - b } }{(C_1{(x-x^*)}^q_{}- B_1{(x - x^*)}^p_{}) (1+ O(|x-x^*|))},
 $$
 here $\sqrt{\varGamma  - b(x) } \gg 1$. 
 
 And thus : 
 
(i) if  $q > p$, $\displaystyle \frac{b_x}{P_x} \sim Cte. \sqrt{\varGamma  - b }$  \quad  (ii) if $q < p$, we have    $\displaystyle \frac{b_x}{P_x} \sim  0$  for  $x \sim x^*$  
\quad (iii) if  $p = q$, we have 
 
 \begin{equation*}
 \frac{b_x}{P_x}= \frac{B_1(1+O(|x -x^*|)) \sqrt{\varGamma  - b }}{(C_1-B_1)(1+O(|x -x^*|)}
 \sim  
 \left\{ 
  \begin{aligned} 
    Cte. \sqrt{\varGamma  - b },  \ \mbox{\em if } \  C_1 \neq B_1  \\
  \pm \infty, \ \mbox{\em if } \  C_1= B_1
  \end{aligned}
  \right.   \ \ \mbox{for}  \ \ x \sim x^*.
\end{equation*}
  The situation would be almost the same as in the case  [II]. 
\begin{flushright}
q.e.d.
\end{flushright}

\medskip

\noindent
  We see in fact: 

(1) \  We have seen above that Airy's shallow water waves propagate almost everywhere 
(see above) on a deep open sea with the sound speed very large with the depth of the 
sea which is, in fact, with the gravity  $g$ normalized to unity
 \begin{eqnarray}  \sqrt{\varGamma  - b(x) }
 \end{eqnarray}
 as it is generally believed among specialists in tunamis study, while 
 \begin{eqnarray}  \left| \frac{\varGamma - H}{\sqrt{\varGamma  - b(x) } }\right|
 \end{eqnarray}
 is sensitively small,  $H$  being the mean depth of the sea at rest. 
 
 \medskip
 
(2) \  We cannot neglect, however, the effect of the variable depth on the propagation 
speed of tunamis even on the deep open sea on certain exceptional points as mentioned
in  [II]  above.  If it is not very visible, it would be because of the followings: 
the ``tunamis" waves are of very great length (in the case of 11.3.2011  tunamis, they 
are about $500$ kilometers  east/west  and $800$ kilometers  north/south) even if the
amplitude is not so great (the maximum height of source waves of tunamis of  
11.3.2011, Hukushima, is measured as about  $7$  meters by specialist: see Takagawa’s 
paper  (c)   cited on the top of this paper), and thus the mass of each wave of tunamis 
would be really enormous and the big inertia of it could cancel this possible rushing 
of ``tunamis" waves on the deep open sea. Tunamis are not, once more, Cheshire cat!  
Adding to this, the sound speed of ``tunamis" propagation is already very large on the 
deep open sea. 

\medskip

{\bf  Example}   For a better understanding of the phenomena stated in  (2)  above, I give 
an example of a figure of the water bottom and possible fluid (water) velocity as a particular case, supposing that  $u_x = 0$ and also $\varGamma_x - b_x = 0$  for $ x_1 < x < x_2$:

{\em  By the water bottom $y=b(x) = B_0 + B_1 x$, $ x_1 < x < x_2$, $B_j$ $j=1,  2$ being constants, the tunamis on the water flow with the velocity $u(t, x) = a(t) = A_0 - B_1t$, $u_x$ vanishing, would have an affection on the propagation speed by $\pm\infty$ for $B_1 >0$ and $B_1 < 0$, respectively, on this interval. }

We see in the first place that the system of tunamis equations  (3.5)--(3.6)  becomes 
\begin{eqnarray}  && \gamma_t^{} + (\gamma +a(t)) \gamma_x = -\frac{1}{2}(a^\prime(t) + b^\prime(x))   \\
&&
 \gamma_t^{} - (\gamma -a(t)) \gamma_x = \frac{1}{2}(a^\prime(t) + b^\prime(x)) 
 \end{eqnarray}
 and in the sequel  $\gamma$ must satisfy 
 \begin{eqnarray} && \gamma_t^{} + a(t) \gamma_x^{}  = 0.
 \end{eqnarray}
  On the other hand, by subtracting  (3.10)  from  (3.9),  we have 
\begin{eqnarray} && (\gamma^2)_x^{} = - (a^\prime(t) + b^\prime(x)), \mbox{i.e.} \ 
 \gamma^2 = c(t) - a^\prime(t) x -b(x), \ for \  some \  c(t).
  \end{eqnarray}
  We have in the sequel 
\begin{eqnarray} && \varGamma = \varGamma(t, x) = c(t) -a^\prime(t) x, \quad   x_1 < x < x_2.  
\end{eqnarray}
Now, the condition $\varGamma_x - b_x = 0$  gives 
\begin{eqnarray} && \varGamma_x(t, x)- b_x(x) = \varGamma(t, x) = c(t) -a^\prime(t) x, \quad   x_1 < x < x_2.  
\end{eqnarray}
and thus we see, 
\begin{eqnarray} &&  -a^\prime(t) = b^\prime(x) = B_1: constant, \quad   x_1 <  x  < x_2.  
\end{eqnarray}
and, in consequence, we have 
\begin{eqnarray} &&  u(t, x) = a(t) =A_0 -  B_1 t,  \  b(x) = B_0 + B_1 x,  \quad   x_1 <  x  < x_2.  
\end{eqnarray}
 Finally $\gamma$ above must satisfy the equation  (3.11)  and thus we have  
$c(t)=c_0$: constant.  We see in fact that
\begin{eqnarray} &&  \gamma = \sqrt{c(t) - a^\prime(t) x - b(x)} = \sqrt{c(t) - B_0}
\end{eqnarray}
satisfies that 
\begin{eqnarray} &&  \gamma_t^{} + a(t) \gamma_x^{} = \frac{c^\prime(t)}{2 \sqrt{c(t) - B_0} }= 0, 
\end{eqnarray}
that is, $ c^\prime(t) =0$. 

This gives us 
\begin{eqnarray} &&  \varGamma  =  \gamma^2 + b(x) = c_0 + B_1 x, \quad   x_1 <  x  < x_2.
\end{eqnarray}

\section{Real tunamis: tunamis {\em P} explose in approaching the coasts}

    \subsection{ }  We analyze tunamis on the coasts. Speaking like this is, however,  
problematic mathematically. ``On the coast" means the possible apparition of dry 
boundary, and this would  be critical for the existence theorem for Euler equations 
for water waves  (2.1)--(2.4);  we analyze in fact mathematically the effects of the diminution 
of the depth of water on tunamis  $P$  approaching the coasts: $\gamma^2=\varGamma-b(x)\sim 0$,
implying the transformation of ``gentle" tunamis on a deep open sea to a violent precipitations 
of rushing tunamis on the coasts.  That is  $\sqrt{\varGamma-b(x)} \gg 1$ is no more the cruising 
speed for  $P$    approaching a ``sloping beach":
 {\em  an analysis of shallow water waves of Airy on a shallow water on the sea-bed satisfying $b_x(x) > 0$  with non-vanishing $b_{xx}$ {\em (}possibly near the coasts{\em )}}.

\medskip

   We stay always in mathematics with open sea covering entirely the domain of 
analysis, but we discussed a phenomenology of real tunamis with the coast where  
$b_{xx}$ vanishing: ``sloping beach" in details. My Note in  ``2017 Annual Reports in 
ESI (Vienna)". 

\medskip

{\bf Remark}  \  We give here an example of possible ``{\em safe}" seabed: $ y = b(x)$.
Let $b(x)$  be:
\begin{eqnarray} && b(x)= -h + K(\tanh x -1), K \ and \ h \ being \  positive \  constants.
\end{eqnarray}

We see :
\begin{eqnarray*}
b(+\infty )=-h,\ b(-\infty ) = - h - 2 K, \ b(0)= -  h -  K,  \\
            b_x > 0, - \infty < x  < +\infty ; \quad b_{xx} > 0 , x < 0;  \quad b_{xx} < 0 , x > 0;   \quad  b_{xx}(0)=0.
  \end{eqnarray*}

   In fact, tunamis means originally ``attacking sea waves on sea-ports" in japanese. 
They happen thus {\em not on the deep open sea but on the shallow water}, very particular 
destructive feature of Airy's shallow water waves arriving on the coasts.  They
lose on the coast of shallow water their original cruising propagation sound speed 
so great on the deep open sea.  So, if mathematics want to show the mechanism 
of so great destructive tunamis on the coasts, they must clarify who and how he 
transforms so great kinetic energy of tunamis on the  deep open sea 
to that so great destructive rushing tunamis's energy {\em on the coasts}. 

\medskip

   We study now the tunamis  $P$ and  $Q$  satisfying
     \begin{eqnarray} 
  &&  P_t  + (\gamma+u+\frac{b_x}{P_x} )P_x = 0 \\
    && Q_t - (\gamma-u-\frac{b_x}{Q_x} )Q_x =  0,
 \end{eqnarray}
 with $P(0, x)= p(x)$ for coastal tunamis $Q$  with $Q(0,x))= 0$, on such a shallow sea as
$\gamma^2=\varGamma-b(x) \ll 1$  when it occurs that $P_x(t, X)= 0$  with $b_x(X) \neq 0$.  

\medskip

Phenomenologically speaking, the following two situations are in order: 

 \medskip
 
[1] \  Under the condition $\gamma=\sqrt{\varGamma-b(x)} \ll 1$, on $x=X$,  {\em before the crest},
we have \newline $P_x(t,X-\epsilon) > P_x(t, X)=0 $ for  $\epsilon \ll 1$ and thus $P_x(t, X-\epsilon)\to +0$
as $\epsilon \to 0$. It implies that our tunamis $P(t, x)$  get at  $x=X- 0$ an instantaneous $+\infty$
propagation speed rushing thus as inland tunamis. 

 \medskip

[2]  \ Under the condition $\gamma = \sqrt{\varGamma - b(x)} \ll 1$, on $x=X$,  {\em after the trough}, 
we have $P_x(t, X-\epsilon) < P_x(t, X)=0 $ for  $\epsilon  \ll 1$ and thus $P_x(t, X-\epsilon)\to  - 0$
as $\epsilon \to 0$. It implies that our tunamis $P(t, x)$  get at  $x=X- 0$ an instantaneous $-\infty$
propagation speed rushing consequently thus to the outer sea as offshore tunamis.

\medskip

     {\bf Remark}  These are not simple repetition of  (A)  and  (B)  in [II] of the  $\S$3, 
      as mentioned in  ($4^\circ$)  there.
 
 \medskip
 
  We begin now to show how we can "find" the above mentioned   $x = X$ {\em  before crest}
or {\em after trough}.  We know that    
   $\displaystyle P_x = u_x + \frac{\varGamma_x - b_x}{\sqrt{\varGamma  - b }} =0 $  and 
thus we see for  $x=X$,  $P_x(t, X)= 0$  implies that
$-u_x(X)\sqrt{\varGamma-b(X)}=\varGamma_x(X)-b_x(X) \sim \pm 0 $ for the shallow sea where   
$\gamma=\sqrt{\varGamma-b(x)} \ll 1$. In fact, 

 \smallskip
 
(i) near the crest where  $u_x  < 0$, we have   $\varGamma_x(X)-b_x(X)\sim +0, \varGamma_x(X) - b_x(X) > 0$;

 \smallskip

(ii) near the trough where  $u_x  > 0$, we have   $\varGamma_x(X)-b_x(X)\sim -0, \varGamma_x(X)-b_x(X) < 0$. 


\begin{prop}   On the shallow water near the coasts or the shores, if our tunamis 
have  $P_x(X) = 0 $  for   $x = X$,  it would be on  X  before and near the crest or on  X
after and near the trough, on which the tangent of the seabed and that of the tunamis
waves are very close. 
\end{prop} 
 
  {\bf Remarks}
  
  
(a) \  It seems that we have seen, in 2011 Tohoku tunamis, explosions of tunamis 
waves just on the shoreline, but this Proposition shows that it was not on the shoreline,
the border of the seawater and the coastland, they are rather stimulated by the sea bed
before shoreline who has the same tangent with the approaching tunamis waves. 

\medskip

 (b) \  Phenomenologically speaking and in real world, on the shore: 
 $ \varGamma_x(X)-b_x(X) \sim 0$, 
it must be  $-u_x(X)\sqrt{\varGamma-b(X)}=\varGamma_x(X)-b_x(X) =0 $.   Mathematically, however, 
it is problematic with the dry boundary to have the existence theorem for water surface for 
Euler equations. 

\medskip

Now, we examine the behavior of   $P_x(X)-P_x(X-\epsilon)$   near the crest and the trough 
above in the Proposition 4.1  when it occurs  $P_x(X)= 0$  .   As we showed before, we have

\medskip
 $\displaystyle P_x(X)-P_x(X-\epsilon)=u_x(X)-u_x(X-\epsilon) +
\frac{(\varGamma_x - b_x)(X)}{\sqrt{(\varGamma  - b)(X)}} -
\frac{(\varGamma_x - b_x)(X-\epsilon)}{ \sqrt{(\varGamma - b)(X-\epsilon)} } 
=\epsilon u_{xx}(X-\delta\epsilon)+ + \epsilon\frac{(\varGamma_x - b_x)_x(X-\delta^{\prime}\epsilon)}{\sqrt{(\varGamma  - b)(X-\delta^{\prime}\epsilon)} }
-\frac{\epsilon}{2} \left[\frac{1}{\sqrt{(\varGamma  - b)(X-\delta^{\prime}\epsilon)}} \left( \frac{(\varGamma_x - b_x)(X-\delta^{\prime}\epsilon)}{\sqrt{(\varGamma  - b)(X-\delta^{\prime}\epsilon)}} \right)^2\right]$, 
for $0<\delta, \delta^\prime<1$. 

\medskip

Then, we have

\begin{prop}    Suppose that our tunamis  $P$  satisfy $P_x(X)= 0$   on  X  in the 
Prpoposition 4.1, we have

{\em (1)}  On the crest side, where  $u_x< 0$: if our tunamis  $P$  is rather ``static", that is the 
decreasing of  $u$  slows down: $(u_x)_x= u_{xx} < 0$, then  $P$  get an instantaneous  
$+\infty$  propagation speed at  $x=  X- 0$,  rushing thus as inland tunamis. 

{\em (2)}  On the trough side, where  $u_x > 0$: If our tunamis  $P$  are rather ``dynamic", that 
is increasing of  $u$  grows up: $ (u_x)_x=u_{xx}  > 0$  and ${(\varGamma_x(x)-b_x(x))}_x > \frac{1}{2}|u_x(X)|^2$ 
for  $X-\epsilon < x < X$,  then  $P$  get at   $x = X-0$  an instantaneous $-\infty$ speed rushing to the outer sea as open sea tunamis. 
\end{prop}
\begin{proof} We see first that 
 
 [\#1] At $x=X$, we have  
 $$\displaystyle \frac{\varGamma_x(X-\delta\epsilon)-b_x(X-\delta\epsilon))}{\sqrt{(\varGamma(X-\delta\epsilon)-b(X-\delta\epsilon)}}=
 -u_x(X) +O(\epsilon)$$  
 by the fact that
 $\displaystyle \frac{\varGamma_x(X)-b_x(X)}{\sqrt{(\varGamma(X)-b(X)}} -\frac{\varGamma_x(X-\delta\epsilon)-b_x(X-\delta\epsilon))}{\sqrt{(\varGamma(X-\delta\epsilon)-b(X-\delta\epsilon)}}  =O(\epsilon)$  and 
   $\displaystyle P_x = u_x + \frac{\varGamma_x - b_x}{\sqrt{\varGamma  - b }} =0 $.
   
   \medskip
   
   [\#2] And then, 
   
    $\displaystyle P_x(X)-P_x(X-\epsilon) 
=\epsilon u_{xx}(X-\delta\epsilon)+\epsilon\frac{(\varGamma_x - b_x)_x(X-\delta^{\prime}\epsilon)}{\sqrt{(\varGamma  - b)(X-\delta^{\prime}\epsilon)} } 
-\frac{\epsilon}{2} \frac{ \left|u_x(X)\right|^2}{\sqrt{(\varGamma  - b)(X-\delta^{\prime}\epsilon)}}  +O(\epsilon^2)$.
\medskip

   As we showed in the Proposition 4.1, the  point  $x = X$  on which   $P_x(X)=0$
occurs would be (i) near and before the crest or  (ii) near and after the trough.
Then we are going to, now, see exactly the conditions for the right-hand side 
of  [\#2]  above to be positive or negative.    

\medskip

(i) the crest side: $u_x  < 0$ .  If  $u_{xx}  < 0 $, then $ u_x$  is a decreasing function, that is $u_x$   
is decreasing as  the current advances. And thus, our tunamis waves are rather amassing making 
the crest justly.  We see now that the right-hand side of [\#2] above is negative, because, in fact, 
$(\varGamma_x-b_x)_x(X- \delta^{\prime}\epsilon) < 0$. Thus we have  $P_x(X-\epsilon) > P_x(X) = 0$
which implies  $P_x(X-\epsilon) \to +0, \epsilon \to +0 $   and our tunamis  $P$  get,  at $x= X -0$   an
instantaneous $+\infty $ propagation speed rushing, in consequence, as inland tunamis.

\medskip

(ii) the trough side: $u_x > 0$.  
If  $ u_{xx} > 0$, $u_x$  is an increasing function and the increment of  $u$  grows up.
That is, our tunamis  $P$  would be going to heave up: 
we see   $(\varGamma_x- b_x)_x (X- \delta^{\prime}\epsilon) > 0$  and also the fact that we have
 $((\varGamma_x- b_x)(x))_x > \frac{1}{2} |u_x(X)|^2$  for  $X-\epsilon < x < X $  shows that the right-hand side of [\#2] to be positive. Thus we have $P_x(X)-P_x(X-\epsilon)>0$  which gives
 $0 = P_x(X) > P_x(X-\epsilon)$  and  $P_x(X-\epsilon) \to - 0, \epsilon \to +0$. Our tunamis $P$ get 
 at  $x=X-0$  an instantaneous $-\infty $  speed of propagation rushing to the open sea as offshore  tunamis. 
 \end{proof}
 
 
  {\bf Remark 4.3}  Airy's shallow water waves would, thus, break down violently ``on a 
gentle sloping beach" rushing on shore by the sudden gain of plus infinity propagation
speed just before the crest and minus infinity propagation speed just after the trough. 
This ``possible" breakdown constitute effectively the violent rushing inland tunamis.
The mathematically exact analysis of this wave-breaking itself, however, would be out 
of the reach of this Proposition. 

\medskip 
  
   {\bf Remark 4.4}  The famous engraving of Hokusa\"i is far from the tunamis. It represents
well, nevertheless, the features of shallow water waves near the coasts. Of course, that 
is Hokusai himself who seized them well. On the tableau, the crest, advancing toward 
the coast, is already overhanging to rush forward, perhaps, violently. Or, to break down 
there... The troughs, on the other hand, are advancing toward the outer sea, as is shown 
by the bows of vessels. It would be not impossible that bottoms of some of these vessels 
are almost rubbing the shallow sea bed there.

\bigskip

\begin{center} 
\includegraphics[width=13truecm]{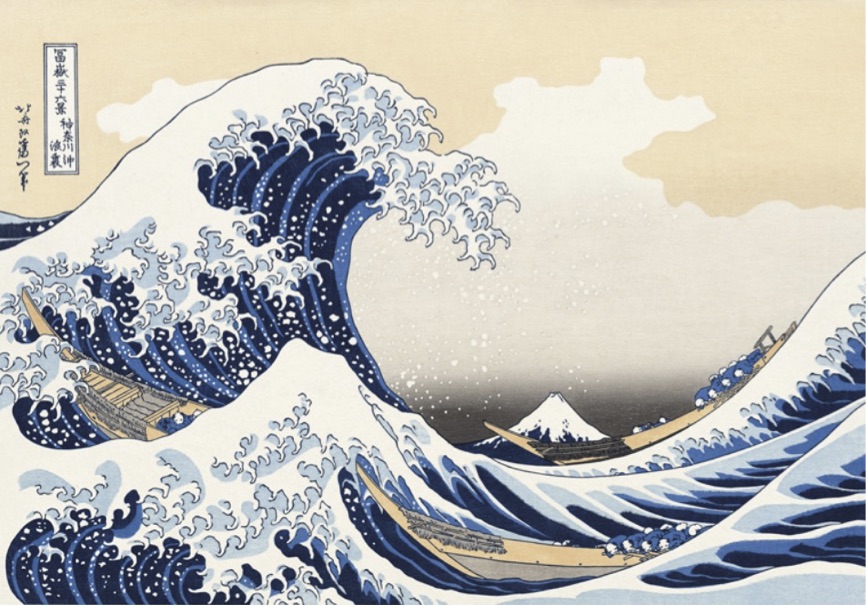}
\end{center}
 
\medskip 

\subsection{ Appendix} 
 
  We give now an application of our Proposition 4.1  which we state again as:
  
  \begin{prop}   On the shallow water near the coasts or the shores, if our tunamis 
have  $P_x(X) = 0 $  for   $x = X$,  it would be on  X  before and near the crest or on  X
after and near the trough, on which the tangent of the seabed and that of the tunamis
waves are very close. 
\end{prop} 
 
 Tunamis have a very long wavelength, from several hundreds to a thousand kilometers 
with rather small amplitudes, and one could easily feel them  as rather gentle
long waves when  they see them from the  coasts.  In 2011, in fact, there were so many people near 
the beaches who ignore alarms from authorities to run away from tunamis disasters and   observe  tunamis arriving from the outer sea so "majestically". Not a few people were taken away by a 
so sudden arriving and attacking of violent tunamis. A cameraman, who was taking movies of tunamis 
on the supposed shoreline, confessed  that he could not say in which moment he started to 
film these sudden rushing violent tunamis. It would be rather naturel if we refer our Proposition 4.2. 
 
 Now, we propose to make a system of alarm so that  people know enough what are the exact situations 
  of tunamis arriving to coasts.

   We knew now, by our Proposition 4.1, how we can "find" the above mentioned 
   $x = X$ {\em  before crest} or {\em after trough}.  We know that  
   $\displaystyle P_x = u_x + \frac{\varGamma_x - b_x}{\sqrt{\varGamma  - b }} =0 $  and 
thus we see for  $x=X$,  $P_x(t, X)= 0$  implies that
$-u_x(X)\sqrt{\varGamma-b(X)}=\varGamma_x(X)-b_x(X) \sim \pm 0 $ for the shallow sea where   
$\gamma=\sqrt{\varGamma-b(x)} \ll 1$. 

 \smallskip
 
(i) near the crest where  $u_x  < 0$, we have   $\varGamma_x(X)-b_x(X)\sim +0, \varGamma_x(X) - b_x(X) > 0$;

 \smallskip

(ii) near the trough where  $u_x  > 0$, we have   $\varGamma_x(X)-b_x(X)\sim -0, \varGamma_x(X)-b_x(X) < 0$. 

  Now, once we know these    $x = X$ {\em  before crest} or {\em after trough}, we trace a 
perpendicular line to the coast in question passing by this point  $X$, and calculate the tangent 
of the sea bed on this line.  It is not difficult as we have precise seabed maps published by 
The Japan Coast Guard .
  Next, we fly a Drone over the tunamis waves in questions and  take photos of tunamis
waves to calculate tangents of these tunamis waves along the same  perpendicular line to the coast
by some AI.
  If these two tangents satisfy conditions: 
  $-u_x(X)\sqrt{\varGamma-b(X)}=\varGamma_x(X)-b_x(X) \sim \pm 0 $ ,
we see the situation would be dangerous.   One must   give a emergency alert to people for escaping  
from the coasts to higher places immediately by making a possible alarm device.  


\bigskip
 
 {\bf Acknowledgements} 
 The author expresses his hearty gratitudes to Taro Kakinuma and 
Tomohiro Takagawa who clarified him many aspects on tunamis. 
Kakinuma gave him these three papers shown in the beginning of 
this paper. The work of Takagawa gave him especially the confirmations 
of his aspects on tunamis, --tunamis as Airy's shallow water waves-- 
a strong support by the real analysis of size and structures of 3.11 Tohoku 
tunamis mathematically and also from the point of view of coast guard. 
A. Matsumura pointed out the incorrect ``proofs" of precedent manuscripts 
by several resumptions to arrive at this audible one. The author is much 
obliged to him.  The author is much obliged also to Y. Sakane  who has given
audible style of actual manuscript.This work was supported by the Research 
Institute for Mathematical Sciences, an International Joint Usage/Research 
Center located in Kyoto University.

\end{document}